\documentclass[12pt,a4paper]{article}
\usepackage{amsmath}
\usepackage{amsthm}
\usepackage{amssymb, enumerate}
\usepackage{amsbsy}
\usepackage{stmaryrd}
\usepackage{amsfonts}
\usepackage{amstext}
\usepackage{amscd}
\usepackage[dvips]{graphicx, psfrag, epsfig}
\usepackage{hangcaption}

\numberwithin{equation}{section}
\theoremstyle{plain}
\newtheorem{thm}{Theorem}[section]
\newtheorem{prop}[thm]{Proposition}
\newtheorem{cor}[thm]{Corollary}
\newtheorem{lem}[thm]{Lemma}
\theoremstyle{definition}
\newtheorem{exa}[thm]{Example}

\newtheorem{rem}[thm]{Remark}
\newtheorem{defi}[thm]{Definition}

\DeclareMathOperator*{\real}{\mathbb{R}}

\DeclareMathOperator*{\comp}{\mathbb{C}}
\DeclareMathOperator*{\nat}{\mathbb{N}}

\newcommand{\nc}{\mathcal{NC}(n)}

\newcommand{\mpn}{\mathcal{M}(n)}
\newcommand{\lpn}{\mathcal{LP}(n)}

\newcommand{\cseries}{\comp \llbracket z_1,\cdots,z_r \rrbracket}

\newcommand{\I}{\text{\normalfont Id}}
\newcommand{\mulb}{\text{\normalfont Mul}^r \llbracket \mathfrak{B} \rrbracket}
\newcommand{\smiley}{\mathop{\kern0.25em\buildrel{.\,.}\over{\rule{0em}{0.2em}} \kern-1.05em\bigcirc \kern-0.73em\scriptscriptstyle\smallsmile\kern0.3em}}

\begin{document}
\title{On operator-valued monotone independence}
\author{Takahiro Hasebe\footnote{Graduate School of Science,  Kyoto University,  Kyoto 606-8502, Japan. 
Email: thasebe@math.kyoto-u.ac.jp, Tel: 075-753-3700, Fax: 075-753-3711. Current address is: Laboratoire de Math\'ematiques, Universit\'e de Franche-Comt\'e, 25030 Besan\c{c}on cedex, France.}   and  Hayato 
Saigo\footnote{Nagahama Institute of Bio-Science and Technology,  Nagahama 526-0829, Japan. Email: h\_saigoh@nagahama-i-bio.ac.jp, Tel: 0749-64-8171, Fax: 0749-64-8140} 
}
\date{}

\maketitle

\begin{abstract}
We investigate operator-valued monotone independence, a noncommutative version of independence for conditional expectation. First we introduce operator-valued monotone cumulants 
to clarify the whole theory and show the moment-cumulant formula. As an application, one can obtain an easy proof of Central Limit Theorem for operator-valued case. 
Moreover, we prove a generalization of Muraki's formula for the sum of independent random variables and a relation between generating functions of moments and cumulants.
\end{abstract}

\section{Introduction}
Noncommutative probability is an algebraic generalization of (Kolmogorovian) probability theory and quantum theory.
A noncommutative probability space is a pair of a $\ast$-algebra $\mathcal{A}$ and a state $\varphi$ on $\mathcal{A}$ 
(i.e.\ a linear functional $\varphi:\mathcal{A} \to \comp$ with the positivity $\varphi(a^\ast a) \geq 0$ for $a \in \mathcal{A}$.) 
Given a probability space $(\Omega, \mathcal{F},P)$, we can associate a pair of a commutative $\ast$-algebra $L^\infty(\Omega,\mathcal{F}, P)$ and a state $E$ which is an expectation regarding $P$. 
In quantum theory, $\mathcal{A}$ is called ``observable algebra''. Noncommutativity of the algebra implies many important physical consequences such as uncertainty principle.

One of the most striking features of noncommutative probability is that the concept of independence is not unique. In addition to the usual notion of independence in probability theory (``tensor independence''), 
free, Boolean and monotone independences were introduced in \cite{V1}, \cite{S-W} and \cite{Mur3}, respectively. 
These four notions can be characterized by natural properties~\cite{Mur4}. 

Moreover, we can also consider a noncommutative version of conditional independence, replacing a state with a conditional expectation taking values in a possibly noncommutative algebra. 
This kind of independence is called operator-valued independence. For instance, the reader is referred to \cite{Dyk1,Spe3,V3} for the free case, \cite{Leh1,Popa2} for the Boolean case and \cite{Popa,Ske} for the monotone case. 
 
Among the four independences, monotone independence is in particular difficult to treat since the concept of ``mutual independence'' fails to hold. More precisely, random variables $Y$ and $X$ may not be 
independent even if $X$ and $Y$ are independent. So we have to distinguish $(X,Y)$ and $(Y,X)$.  On the other hand, monotone independence plays crucial roles in some situations. 
For instance, monotone independence is essential to understand a relation between a free L\'{e}vy process and a classical Markov process, as shown by Biane~\cite{Biane} implicitly 
(Franz pointed out this relation explicitly in \cite{Fra3}).

In this paper, we develop a theory on operator-valued monotone independence for multivariate random variables, generalizing results of Popa~\cite{Popa} for a single random variable. We  
extend the Muraki formula, which describes the moments of the sum of independent random variables, to the operator-valued setting.  

First we define the notion of an operator-valued version of ``generalized cumulants'' to clarify the whole theory and prove the
 moment-cumulant formula, following the idea of \cite{H-S2}.  We apply this to the central limit theorem, to obtain a new expression of the limit distribution. 
 Then we investigate generating functions of moments and cumulants. To this end, we extend the algebraic structure of the ring of multivariate formal power series, focusing on the coefficients of series. Finally we prove 
the extension of the Muraki formula and a differential equation involving moments and cumulants.

\section{Operator-valued monotone independence}
\subsection{Preliminary concepts}
Involutions on algebras are not essential in the scope of this paper, so we do not consider them below. 
In this paper, $\mathfrak{B}$ denotes a unital algebra and $\mathcal{A}$ a unital algebra containing $\mathfrak{B}$ as a subalgebra. We assume that the unit of $\mathfrak{B}$ coincides with that of $\mathcal{A}$. In this paper, algebras can be considered over any commutative field such as $\real$ and $\comp$. We say that $\mathcal{C}$ is a subalgebra of $\mathcal{A}$ over $\mathfrak{B}$ if $\mathcal{C}$ is a subalgebra of $\mathcal{A}$ and $bc \in \mathcal{C}$ for all $b \in \mathfrak{B}$ and $c \in \mathcal{C}$. A subalgebra of $\mathcal{A}$ over $\mathfrak{B}$ may not contain the unit of $\mathcal{A}$. 

For $X_1,\cdots,X_r \in \mathcal{A}$, let $\mathfrak{B}\langle X_1,\cdots,X_r \rangle_0$ denote the subalgebra of $\mathcal{A}$ over $\mathfrak{B}$ consisting of finite sums of elements of $\{b_1X_{i_1}b_2\cdots X_{i_n}b_{n+1}: b_i \in \mathfrak{B}, n \geq 1, i_1,\cdots,i_n \in \{1,\cdots, r \}\}$.  Note that, in general, $\mathfrak{B}$ is not contained in $\mathfrak{B}\langle X_1,\cdots,X_r \rangle_0$. 

Let $\mathfrak{D}$ be another unital algebra containing $\mathfrak{B}$ as a subalgebra. 
A map $f$ from $\mathcal{A}$ to $\mathfrak{D}$ is called $\mathfrak{B}$-linear if $f(b_1 x b_2 + y) = b_1 f(x) b_2 + f(y)$ for all  $b_1,b_2 \in \mathfrak{B}$ and $x,y \in \mathcal{A}$.  A $\mathfrak{B}$-linear map $h$ is called a $\mathfrak{B}$-homomorphism if $h(xy) = h(x)h(y)$ for any $x,y \in \mathcal{A}$. 
A $\mathfrak{B}$-linear map $\varphi$ with values in $\mathfrak{B}$ is called a \textit{conditional expectation} if $\varphi(b) = b$ for $b \in \mathfrak{B}$. From now on we assume that $\varphi$ is a conditional expectation in the above sense. A triple $(\mathcal{A},\mathfrak{B},\varphi)$ is called an \textit{algebraic probability space} or a \textit{noncommutative probability space}, as in the case $\mathfrak{B}=\comp$.

A \textit{random variable} is an element of $\mathcal{A}$ and a \textit{random vector} or vector-valued random variable is an element of $\mathcal{A}^r$ for an $r \geq 1$.

The concept of (joint) moments has to be generalized, since $\varphi(X_{i_1}\cdots X_{i_n})$ $(1\leq i_1,\cdots, i_n \leq r)$ are not sufficient to know the information on a conditional expectation. 

\begin{defi}
Let $X=(X_1,\cdots, X_r)$ be a random vector and $i_1,\cdots, i_{n} \in \{1,\cdots, r\}$ for $n \geq 1$. The multilinear functional $\mu^X_{i_1,\cdots,i_{n}}$ defined by 
\[
\mu^X _{i_1,\cdots,i_{n}}(b_1, \cdots, b_{n}) = \varphi(b_1X_{i_1}b_2\cdots b_nX_{i_{n}}) 
\]
is called an $(i_1,\cdots, i_{n})$-moment of $X$. 
\end{defi}

The monotone independence over $\mathfrak{B}$ was introduced by Skeide~\cite{Ske}. 
\begin{defi} 
Let $\Lambda$ be an index set equipped with a linear order $<$. 
A family of subalgebras $\left(\mathcal{A}_{\lambda} \right)_{\lambda \in \Lambda}$ over $\mathfrak{B}$ is said to be \textit{monotone independent over $\mathfrak{B}$} if
\[
\varphi (X_1 \cdots X_n)= \varphi(X_1 \cdots X_{i-1}\varphi(X_i) X_{i+1} \cdots X_n)
\]
holds for any $X_i \in \mathcal{A}_{\lambda_i}$ whenever $i$ satisfies $\lambda _{i-1}<\lambda _i$ and $\lambda _i>\lambda _{i+1}$ (one of the inequalities is eliminated when $i=1$ or $i=n$). Independence for random vectors $X_\lambda=(X_{\lambda,1} \cdots, X_{\lambda,k_\lambda})$, $\lambda \in \Lambda$ is defined by considering the subalgebras $\mathcal{A}_\lambda:=\mathfrak{B}\langle X_{\lambda,1},\cdots, X_{\lambda,k_\lambda} \rangle_0$.  
\end{defi} 

Let $\Lambda$ be a linearly ordered set and $r \in \nat,r\geq 1$. Random vectors $ X_\lambda=(X_{\lambda,1}, \cdots,X_{\lambda,r})$, $\lambda \in \Lambda$ are said to be \textit{monotone i.i.d.\ (independent, identically distributed)} if they are monotone independent over $\mathfrak{B}$ and $\mu^{X_\lambda}_{i_1,\cdots,i_n}$ does not depend on $\lambda \in \Lambda$ for any $i_1, \cdots, i_n \in \{1,\cdots,r \}$ and $n \geq 1$.

\subsection{Dot operation}
 We introduce a dot operation, following the paper~\cite{H-S2}.   
\begin{defi} \label{dot}
For every $X \in \mathcal{A}$, let us take copies $\{X ^{(j)}\}_{j \geq 1}$ in an algebraic probability space $(\widetilde{\mathcal{A}}, \mathfrak{B}, \widetilde{\varphi})$ such that: 
\begin{enumerate}[\rm(1)]
\item $X \mapsto X^{(j)}$ is a $\mathfrak{B}$-homomorphism for each $j$; 

\item $\widetilde{\varphi}(X_1^{(j)} X_2^{(j)} \cdots X_n^{(j)}) = \varphi(X_1X_2\cdots X_n)$ for any $X_i \in \mathcal{A}$, $j, n \geq 1$; 

\item the subalgebras $\mathcal{A}^{(j)}:=\{X ^{(j)}\}_{X \in \mathcal{A}}$ are monotone independent over $\mathfrak{B}$.  
\end{enumerate}
We define a \textit{dot operation} $N.X$ as follows: 
\[
N.X = X^{(1)} + \cdots + X^{(N)}
\]
for $X \in \mathcal{A}$ and a natural number $N \geq 0$. We understand that $0.X=0$. 
The dot operation can be extended to random vectors: 
\[
N.X:=(X_1^{(1)}+\cdots+X_1^{(N)},\cdots,X_r^{(1)}+\cdots+X_r^{(N)})
\] 
for $X=(X_1,\cdots,X_r)$. 
We can iterate the dot operation more than once in a suitable space. For instance, the symbol $N.(M.X)$ means the sum $(M.X)^{(1)} + \cdots + (M.X)^{(N)}$ of monotone i.i.d.\ random variables $(M.X)^{(k)}$, $k =1,2,3,\cdots$.  
\end{defi}
For simplicity, $\widetilde{\varphi}$ is denoted by the same symbol $\varphi$ in this paper.   
The above dot operation can be realized in a canonical way in terms of a free product with amalgamation. The construction is similar to the $\comp$-valued case~\cite{H-S2}, so that we do not repeat it. 

An essential property of the dot operation is ``associativity up to a state'', described by the proposition below.

\begin{prop}\label{prop5} 
For random variables $X_1,\cdots,X_n$,  
\[
\varphi\left( ((NM).X_1)\cdots ((NM).X_n) \right) = \varphi\left( (N.(M.X_1))\cdots (N.(M.X_n)) \right).  
\]    
\end{prop}
The proof is quite similar to that of \cite{H-S2}.

\section{Monotone cumulants}
Let us introduce terminologies and notations regarding partitions of a set. The following definitions of partitions and ordered partitions are possible on any linearly ordered set, but we only consider the set $\underline{n}:= \{1,\cdots,n\}$ for simplicity. $\pi$ is said to be a \textit{partition} of $\underline{n}$ if $\pi=\{V_1,\cdots,V_k\}$, where $V_i$ are non-empty, disjoint subsets of $\underline{n}$ and $\cup_{i=1}^k V_i = \underline{n}$. The number $k$ is denoted as $|\pi|$ and an element $V$ of $\pi$ is called a block. A partition $\pi$ is said to be \textit{crossing} if blocks $V,W \in \pi$ exist so that there are elements $a,c \in V$ and $b,d \in W$ satisfying $a<b<c<d$. $\pi$ is said to be \textit{non-crossing} if it is not crossing. The set of the non-crossing partitions of $\underline{n}$ is denoted as $\nc$. A block $V$ of a partition is called an \textit{interval block} if $V$ is of the form $V=\{k,k+1,\cdots, k+l\}$ for $k \geq 1$ and $0 \leq l \leq n-k$. 
The set of the interval blocks is denoted by $IB(n)$. 
A partial order can be defined for partitions. For partitions $\pi$ and $\sigma$, the relation $\pi \leq \sigma$ means that for any block $V \in \pi$, there exists a block $W \in \sigma$ such that $V \subset W$. For instance, the partition consisting of one block $\{1,\cdots,n\}$ is larger than  any other partition.  

In addition to partitions, we need ordered partitions in this paper. An \textit{ordered partition} $\pi$ of $\underline{n}$ is a sequence $\pi=(V_1,\cdots, V_k)$, where $\{V_1,\cdots,V_k \}$ is a partition of $\underline{n}$. The number $k$ is also denoted as $|\pi|$. 
Let us denote by $\lpn$ the set of the ordered partitions of $\underline{n}$. 

We introduce a partial order on blocks in a partition $\pi \in \nc$ as follows: For $V,W \in \pi$ we denote $V \succ W$ if there are $i,j \in W$ such that $i < k < j$ for all $k \in V$. 
Visually, $V \succ W$ means that $V$ lies in the inner side of $W$. For instance, $\{4,5,6 \} \succ \{3,8,10 \}$ in Fig.~\ref{dia1}. We assume that the relation $\succ$ does not include the equality: $V \succ W$ implies $V \neq W$ in this paper. 
A \textit{monotone partition} of $\underline{n}$ is an ordered partition $\pi =(V_1, \cdots, V_k) \in \mathcal{LP}(n)$ which satisfies the following properties: 
\begin{enumerate}[\rm(1)] 
\item $\{V_1,\cdots,V_k\} \in \nc$, 
\item If $V_i \succ V_j$, then $i > j$.  
\end{enumerate}
The set of monotone partitions of $\underline{n}$ is denoted by $\mpn$. 

\begin{figure}
\centering
\includegraphics[width=8cm,clip]{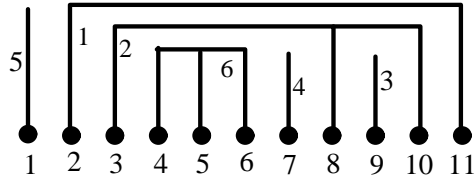}
\caption{A monotone partition $(\{2, 11 \}, \{3, 8, 10 \}, \{9 \}, \{7 \}, \{1 \}, \{4, 5, 6 \})$. Each block is labeled by a number to clarify the order on the blocks.}
\label{dia1}
\end{figure}

The following definition was used by Speicher \cite{Spe3} to describe the moment-cumulant formula for the case of free probability with amalgamation. 
Let $A_n$ be a multilinear functional from $\mathcal{A}^n$ to $\mathfrak{B}$ for $n \geq 1$. A multilinear functional $A_\pi$ for a non-crossing partition $\pi$ of $\underline{n}$ is defined by the recursive relation
\[
A_\pi(X_1,\cdots,X_n)=A_\sigma (X_1,\cdots, X_{k-1}, A_{m+1}(X_k,\cdots,X_{k+m})X_{k+m+1},\cdots,X_n),  
\]
where $V =\{k,\cdots,k+m\}$ is an interval block of $\pi$ and $\sigma$ denotes the non-crossing partition $\pi \backslash \{V\}$ of $\underline{n}\backslash V$.\footnote{While we have not defined partitions of an arbitrary linearly ordered set, a non-crossing partition of $\underline{n}\backslash V$ can be naturally defined by using the linear order structure of $\underline{n}\backslash V$.}   If $\pi=(V_1,\cdots,V_k)$ is an ordered partition such that $\{V_1,\cdots, V_k\} \in \nc$, then we define $A_\pi$ in the above way, neglecting the order structure of $\pi$.

In this paper, $\varphi_\pi(X_1,\cdots,X_n)$ always denotes the above construction arising from the multilinear functionals  
$\varphi_n(X_1,\cdots,X_n) :=\varphi(X_1 \cdots X_n)$.  

The following result is useful to understand an interplay among $\lpn$, $\nc$ and monotone independence. Therefore a detailed proof is presented.  
\begin{lem}\label{lem32}
For each non-crossing partition $\pi$, there exists a polynomial $a_\pi(x)$ which does not contain a constant term such that 
\[
\varphi\Big((N.X_1)\cdots (N.X_n)\Big)= \sum_{\pi \in \mathcal{NC}(n)} a_\pi(N)\varphi_\pi(X_1,\cdots, X_n)  
\]
for all $X_1,\cdots,X_n$ and $N \in \nat$. 
\end{lem}
\begin{proof}
(Step 1) To calculate $\varphi\Big((N.X_1)\cdots (N.X_n)\Big)$, we have to know $\varphi(X_1^{(i_1)}\cdots X_n^{(i_n)})$ for each sequence 
$(i_1,\cdots,i_n)$ of natural numbers. We first prove that $\varphi(X_1^{(i_1)}\cdots X_n^{(i_n)})$ can be written as $\varphi_\sigma(X_1,\cdots,X_n)$ with a $\sigma \in \nc$. 
 Let us associate an ordered partition $\pi =(V_1,\cdots, V_p)$ to a sequence $(i_1,\cdots,i_n)$ as follows. First we define $p_1:= \max \{i_k:1\leq k \leq n\}$ and then $V_1:= \{k: i_k=p_1\}$. Next we define $p_2:=\max \{i_k: k \notin V_1 \}$ and $V_2:= \{k:i_k=p_2\}$. Recursively we define $p_m:=\max \{i_k: k \notin V_1 \cup \cdots \cup V_{m-1} \}$ and $V_m:= \{k:i_k=p_m\}$. Then we obtain an ordered partition $\pi$. 

Next let us define a map $Q$ from $\lpn$ onto $\nc$.

(1) For $\pi=(V_1,\cdots,V_r) \in \lpn$, let us focus on $V_r$ at first. If there exists a block $V_i$ such that $V_r$ and $V_i$ are crossing\footnote{Blocks $V$ and $W$ are said to be crossing if there are $a,c \in V$ and $b,d \in W$ such that $a <b<c<d$ or $d<c<b<a$.}, then we can take the maximal partition $\sigma$ of $V_r$ such that no block of $\sigma$ crosses $V_i$ (`maximal' is for the partial order $\leq$). Iterating this procedure for every $V_i$ crossing $V_r$, we finally obtain the maximal partition $\sigma_r$ of $V_r$, no block of which crosses the other blocks $V_1,\cdots,V_{r-1}$. 

(2) We define ordered partitions $\pi_k:=(V_1,\cdots,V_k) \in \mathcal{LP}(\cup_{i=1}^k V_i)$. Then we carry out the procedure (1) for $\pi_k$ from $k=r-1$ to $k=1$, to obtain partitions $\sigma_k$. Thus we obtain a non-crossing partition $Q(\pi)$ of $\{1,\cdots,n\}$, gathering $\{\sigma_k\}_{k=1}^r$.  
Figs.\ \ref{dia5}--\ref{dia7} are examples of the map $Q$. 
We can check that $\varphi(X_1^{(i_1)}\cdots X_n^{(i_n)})$ is equal to $\varphi_{Q(\pi)}(X_1,\cdots,X_n)$ if $\pi$ denotes 
the ordered partition associated to $(i_1,\cdots,i_n)$. 

\begin{figure}[htbp]
\begin{center}
\includegraphics[width=8cm,clip]{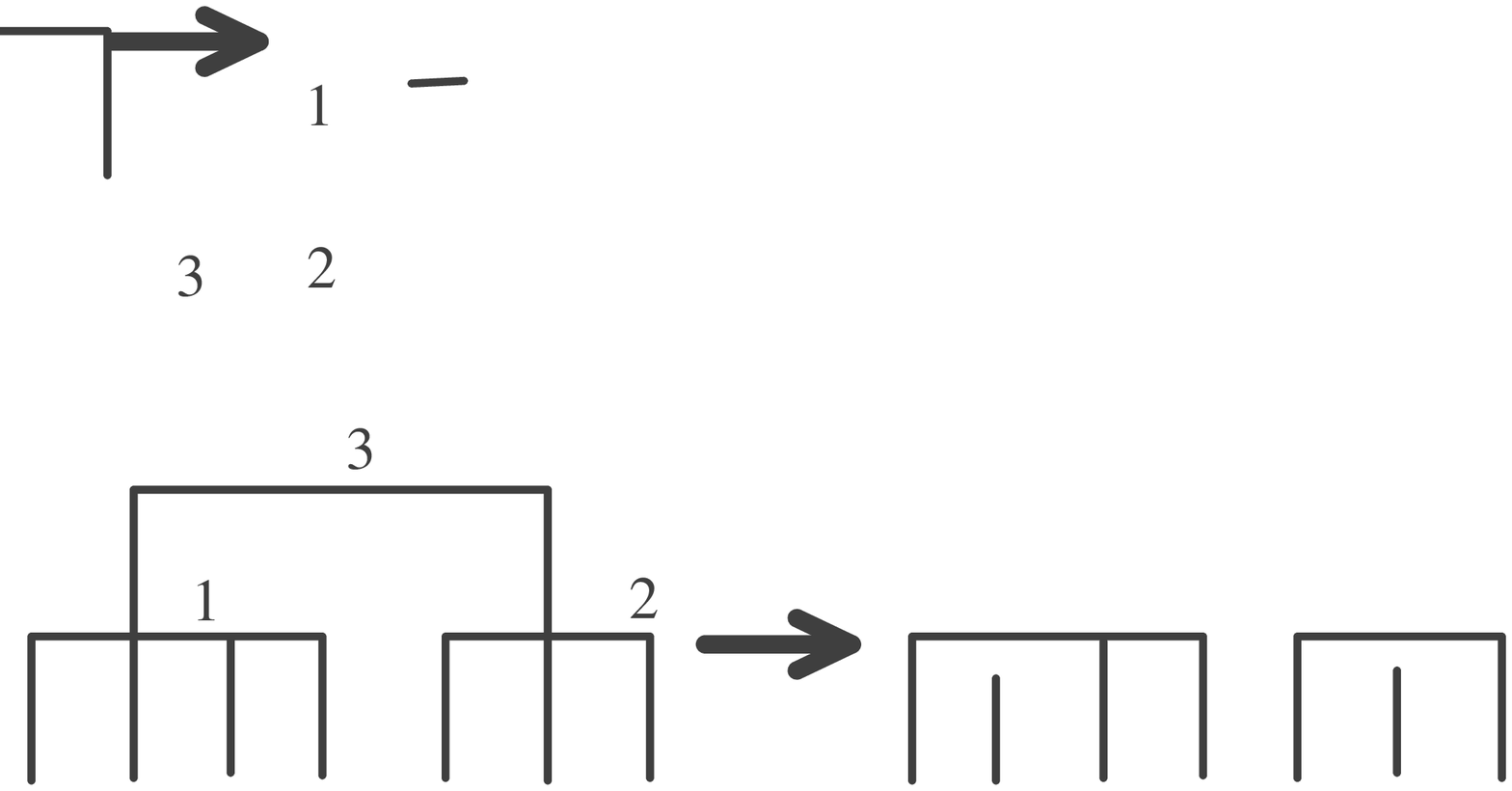}
\caption{The ordered  partition $(\{1,3,4\}, \{5,7 \}, \{2,6 \})$ is mapped to the non-crossing partition $\{ \{1\}, \{2,6\}, \{3,4\}, \{5\},\{7\} \}$.}
\label{dia5}
\includegraphics[width=8cm,clip]{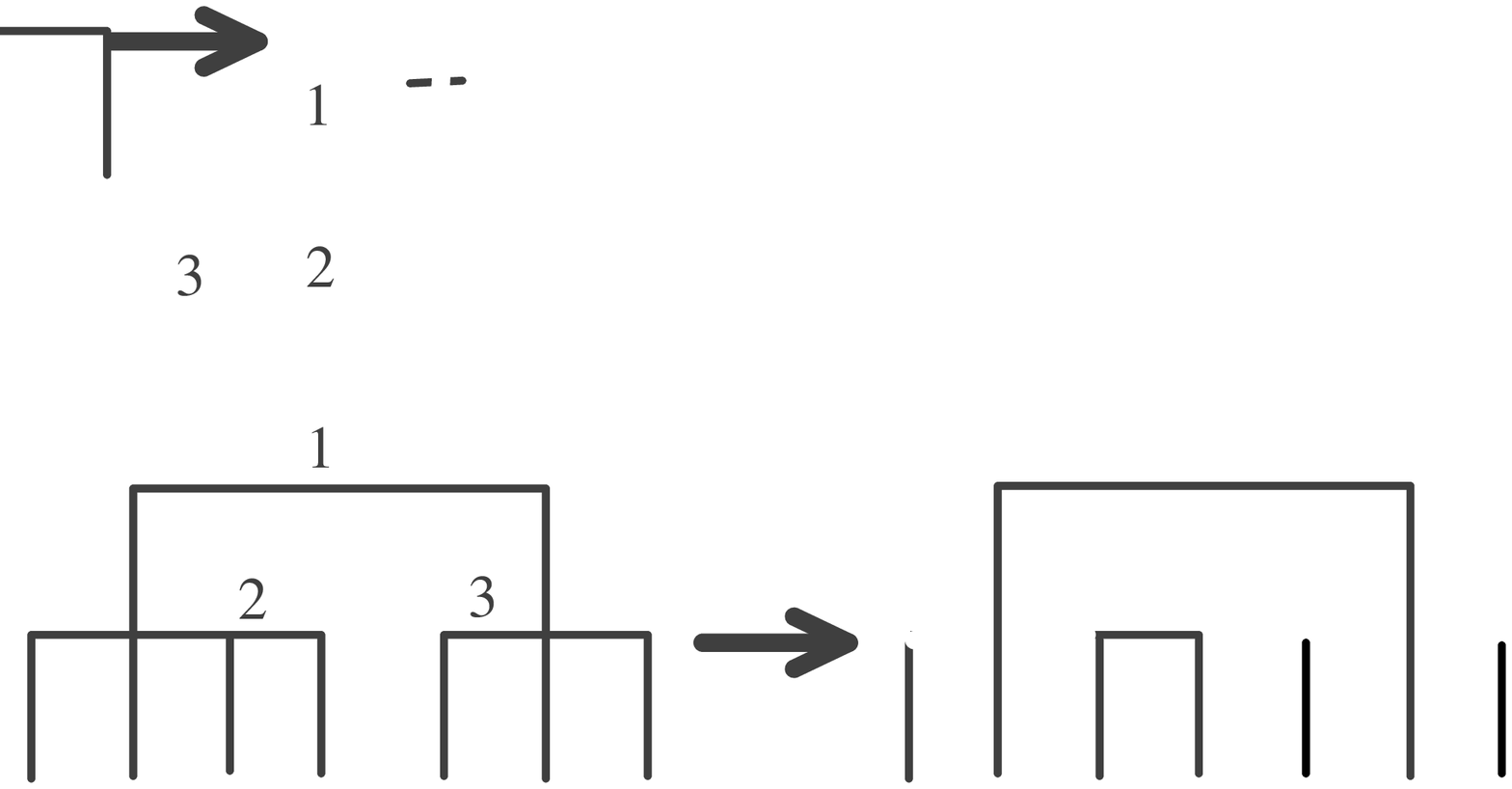}
\caption{The ordered  partition $(\{2,6 \}, \{1,3,4\}, \{5,7 \})$ is mapped to the non-crossing partition $\{ \{1\}, \{2,6\}, \{3,4\},\{5\},\{7\} \}$.}
\label{dia6}\includegraphics[width=8cm,clip]{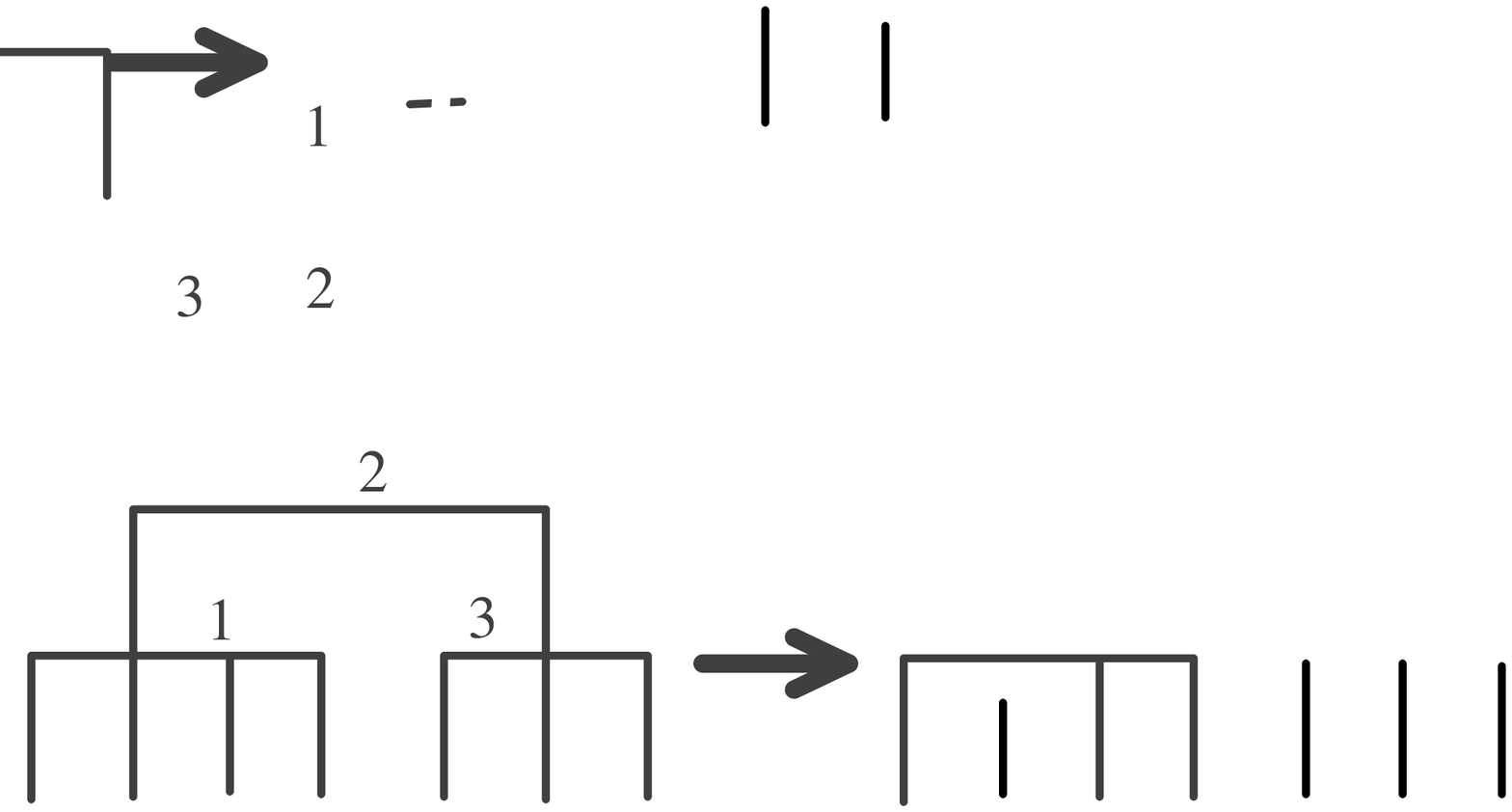}
\caption{The ordered  partition $(\{1,3,4\},  \{2,6 \}, \{5,7 \})$ is mapped to the non-crossing partition $\{ \{1,3,4\}, \{2\}, \{5\}, \{6\},\{7\} \}$.}
\label{dia7}
\end{center}
\end{figure}

(Step 2) The remaining proof is similar to the $\comp$-valued case, based on induction. For $n=1$,  $\varphi(N.X_1) = N\varphi(X_1)$, so that the assertion is true. We assume that the assertion is true for $n$. The identity 
\[
\begin{split}
&\varphi\left((X_1+Y_1)\cdots (X_{n+1} + Y_{n+1})\right) -\varphi(Y_1 \cdots Y_{n+1}) =\sum_{\substack{Z_i\in\{X_i,Y_i\},1 \leq i \leq n+1,\\\text{$Z_j=Y_j$ for some $j$.}}}\varphi(Z_1\cdots Z_{n+1}),
\end{split}
\]
holds for any random variables $X_i, Y_i$. We assume random vectors $(X_1,\cdots,X_{n+1})$ and $(Y_1,\cdots,Y_{n+1})$ are monotone independent over $\mathfrak{B}$. 
Then each term $\varphi(Z_1\cdots Z_{n+1})$ factorizes by using monotone independence and expectations $\varphi(Y_{k}Y_{k+1} \cdots Y_{k+m-1})$ appear satisfying $1 \leq m \leq n$. 
 Let us replace $X_i$ by $X_i^{(1)}$ and $Y_i$ by $X_i^{(2)} + \cdots + X_i^{(M)}$ in $\varphi((X_1+Y_1)\cdots (X_{n+1} + Y_{n+1}))$. 
Then the difference 
$\varphi(M.X_1\cdots M.X_{n+1}) - \varphi((M-1).X_1\cdots (M-1).X_{n+1})$ can be written as 
\begin{equation}\label{eq24}
\sum_{\pi \in \mathcal{NC}(n+1)} b_\pi(M)\varphi_\pi(X_1,\cdots, X_{n+1})  
\end{equation}
for some polynomials $b_\pi(x)$, by using the assumption of induction. 
The above sum is taken over only non-crossing partitions because of the assumption of induction and Step 1. By summing up (\ref{eq24}) over $M$ from $1$ to $N$, we conclude the assertion for $n+1$ since $\sum_{M=1}^N M^{p}$ is a polynomial on $N$ of degree $p+1$ without a constant term. 
\end{proof}

\begin{rem}
The coefficients $a_\pi(N)$ are universal in the sense that they do not depend on a choice of  noncommutative probability spaces $(\mathcal{A}, \mathfrak{B}, \varphi)$ and $(\widetilde{\mathcal{A}}, \mathfrak{B}, \widetilde{\varphi})$ of Definition \ref{dot}. 
\end{rem}
\begin{defi}
We define an $n$th joint cumulant $K_n(X_1,\cdots,X_n)$ as the coefficient of $N$ appearing in $\varphi \Big((N.X_1)(N.X_2)\cdots (N.X_n)\Big)$. $K_n(X_1,\cdots,X_n)$ is also written as $K_n^X$ with $X=(X_1,\cdots,X_n)$ for simplicity.  
\end{defi}

The arguments in \cite[Proposition 5.1, Corollary 5.2, Theorem 5.3]{H-S2} can be easily extended to the operator-valued setting, and one obtains the following.   
\begin{thm}
The following moment-cumulant formula holds: 
\[
\varphi(X_1\cdots X_n) = \sum_{\pi \in \mathcal{M}(n)}\frac{1}{|\pi|!}K_\pi(X_1,\cdots,X_n).  
\]
\end{thm}

From Proposition \ref{prop5} and the proof of \cite[Proposition 3.4]{H-S2}, one can prove the following additivity property of cumulants for monotone i.i.d.\ random variables. 
\begin{prop} For $X=(X_1,\cdots, X_n) \in \mathcal{A}^n$, 
\[
K_{n}^{N.X} = N K_n^X. 
\]
\end{prop}

Central limit theorem was considered in \cite{Popa} and its limit distribution was further studied in \cite{BPV}. 
As an application of cumulants, we can obtain an easy proof and a different formula of the limit distribution using an argument quite similar to that in \cite{H-S}. 
\begin{thm}
Let $(X_i)_{i=1}^\infty$ be monotone i.i.d.\ random variables. We assume that $\varphi(X_i) = 0$ for any $i$ and define $X(N) := \frac{X_1 + \cdots + X_N}{\sqrt{N}}$.  
Then 
\[
\begin{split}
&\lim_{N \to \infty}\varphi (b_1X(N) b_2\cdots b_nX(N)) \\ &~~~~~~~~~~~= \begin{cases}\sum_{\pi \in \mathcal{M}_2(n)}\frac{1}{|\pi|!}K_\pi(b_1 X_1, \cdots, b_n X_1) & \text{if $n$ is even}, \\   
0 & \text{if $n$ is odd},    
\end{cases} 
\end{split}
\]
where 
$\mathcal{M}_2(2k)$ is the set of monotone pair partitions $\{\pi \in \mathcal{M}(2k): |V| = 2 \text{~for any block $V$ of $\pi$} \}$. 
\end{thm}

\section{Generating functions}
A purpose of this paper is to investigate generating functions of moments and monotone cumulants of multivariate random variables $X=(X_1,\cdots,X_r)$. For $\mathfrak{B}=\comp$ and $r=1$, let us consider $M^X(z) = 1+ \sum_{n=1}^\infty \varphi(X^n)z^n$ and $\mu^X(z):=zM^X(z)$. In \cite{Mur3} Muraki proved that $\mu^{X+Y}(z) =\mu^X(\mu^Y(z))$ for monotone independent random variables $X$ and $Y$. 
This relation can also be written as 
\begin{equation}\label{eq64}
M^{X+Y}(z) = M^X(zM^Y(z)) M^Y(z). 
\end{equation} 
In \cite{Popa}, Popa proved an analogue of this formula for general $\mathfrak{B}$ and $r=1$. Muraki's formula was also extended to general $r$ with $\mathfrak{B}=\comp$~\cite{H-S2}. In this paper, we are going to prove the most general version, i.e., for general $\mathfrak{B}$ and $r$. 

This section is also related to the work of Dykema \cite{Dyk1}, in which the ring $\comp \llbracket z \rrbracket$ of formal power series was extended to the operator-valued case. A key in this extension is to replace a formal power series by a sequence of multilinear functionals. Dykema introduced a composition operation and investigated an algebraic structure of such multilinear functionals. 

In this section, we generalize such an algebraic structure to the multivariate case. In other words, we generalize the ring $\comp \llbracket z_1,\cdots,z_r \rrbracket$ generated by free indeterminates to the operator-valued case. Let us consider how to extend the composition of two functions. We generalize not the usual composition operation $\circ$, but a modified associative operation 
\begin{equation}\label{eq8}
(F \bullet G)(z):= F(zG(z)) G(z). 
\end{equation}
This is natural for monotone independence as we can guess from (\ref{eq64}). 

First we extend the ring $\comp \llbracket z_1,\cdots,z_r \rrbracket$ to include the $\mathfrak{B}$-valued case. 
\begin{defi}\label{defi13}
We define the set $\mulb$ of all $F=(F_{i_1,\cdots,i_{n}})_{i_1,\cdots,i_{n} \in \underline{r}, n \geq 0}$, where $F_{i_1,\cdots,i_{n}}$ is a multilinear functional from $\mathfrak{B}^{n}$ to $\mathfrak{B}$. $n=0$ corresponds to a constant $F_{\emptyset} \in \mathfrak{B}$.   
\end{defi}
To define an analogue of the composition, we need some notations and concepts. 
\begin{defi}\label{defi14}
Let $V$ be a subset of $\underline{n}$, denoted as $\{v_{1}, v_2, \cdots, v_{p} \}$ where $1 \leq v_1 < \cdots < v_p \leq n$. We moreover add edges $v_0:= 0$ and $v_{p+1}:=n+1$ for convenience. 
\begin{enumerate}[\rm(1)]
\item The interval blocks $V_i=\{v_{i-1}+1,\cdots, v_{i}-1\}$ for $1 \leq i \leq p+1$ are called the \textit{interpolation blocks of $V$ for $\underline{n}$}.  If $v_{i-1} + 1 = v_i$, we set $V_i: = \emptyset$. If $V=\emptyset$, we consider only one interpolation block $V_1=\underline{n}$.  
Clearly this notion can be extended for the case of any linearly ordered set.
\item Let $V=\{v_1,\cdots,v_p\}$ be a subset of $\underline{n}$ with $v_1 < \cdots < v_p$. For a tuple $(i_1,\cdots, i_{n}) \in \nat^n$, we define $i_V = i(V) := (i_{v_1},\cdots, i_{v_p})$. If $V=\emptyset$, then $i(V):=\emptyset$. For a multilinear functional $F_p:\mathfrak{B}^p \to \mathfrak{B}$ , we set $F_p(b_V):=F_p(b_{v_1},\cdots,b_{v_p})$. 
   \end{enumerate} \end{defi}

\begin{exa}
\begin{enumerate}[\rm(1)]
\item If $V = \{2,3,4,6\} \subset \underline{6}$, the interpolation blocks of $V$ for $\underline{6}$ are given by 
$V_1=\{1\}, V_2=V_3=V_5=\emptyset, V_4=\{5\}$. 
\item Let $\{1,2,3,4,6,7,8\}$ be endowed with the natural order structure.  If $V=\{3,4,7\}$, then the interpolation blocks of $V$ for  $\{1,2,3,4,6,7,8\}$ is given by $V_1=\{1,2\}, V_2=\emptyset, V_3=\{6\}, V_4=\{8\}$. 
\end{enumerate}
\end{exa}   
Now we introduce algebraic structure on $\mulb$ and some operations.    
\begin{defi}\label{defi11}
\begin{enumerate}[\rm(1)]
\item For $F,G \in \mulb$, $F \smiley G \in \mulb$ is defined by 
\[
\begin{cases}
\begin{split}
&(F\smiley G)_{i_1,\cdots,i_{n}}(b_1,\cdots, b_{n})\\&~~~~~~~~~~ :=
\sum_{V=\{v_1,v_2,\dots,v_p\} \subset \underline{n}} F_{i(V)}\big(G_{i(V_1)}(b_{V_1})b_{v_1},\cdots, G_{i(V_{p})}(b_{V_p})b_{v_p}\big)G_{i(V_{p+1})}(b_{V_{p+1}}), 
\end{split}& n \geq 1,\\
(F\smiley G)_{\emptyset}:=F_\emptyset G_\emptyset, 
\end{cases}
\]
where $V_i$ are the interpolation blocks of $V$ for $\underline{n}$. If an interpolation block $V_j$ is empty, then $G_{i(V_j)}(b_{V_j})$ is understood to be $G_\emptyset$. 
The sum over $V \subset \underline{n}$ includes the case $V = \emptyset$; in this case $F_{i(V)}\big(G_{i(V_1)}(b_{V_1})b_{v_1},\cdots, G_{i(V_{p})}(b_{V_p})b_{v_p}\big)G_{i(V_{p+1})}(b_{V_{p+1}})$ is understood to be $F_\emptyset G_{i_1,\cdots,i_n}(b_1,\cdots,b_n)$.  

\item Let $F(t) \in \mulb$ for $t \in \real$ such that each $F(t)_{i_1,\cdots,i_n}$ is differentiable with respect to $t$. Then we define 
$\frac{dF}{dt}(t) \in \mulb$ to be $(\frac{d}{dt}F(t)_{i_1,\cdots,i_n})_{i_1,\cdots,i_n \in \underline{r}, n \geq 0}$. 
\item We define a binary operation $\star$ on $\mulb$ by  
\[ 
\begin{split}
&(F\star G)_{i_1,\cdots,i_n}(b_1,\cdots,b_{n}):= \\&~~~~~~~~~\sum_{V =\{k,\cdots,k+l\} \in IB(n)} F_{i(V^c)}(b_1,\cdots, b_{k-1}, G_{i(V)}(b_V)b_{k+l +1},\cdots,b_{n}). 
\end{split}
\] 
If $k+l = n$, the summand is understood to be $F_{i(V^c)}(b_{V^c})G_{i(V)}(b_V)$. For $n=0$, we define $(F\star G)_\emptyset:=F_\emptyset G_\emptyset$. 
\end{enumerate}
\end{defi}
\begin{rem}
\begin{enumerate}[\rm(1)]
\item The operation $\smiley$ corresponds to the modified composition (\ref{eq8}). Moreover, a relation to the paper \cite{H-S2} is as follows. To treat generating functions related to random vectors with $\mathfrak{B}=\comp$, we use the formal power series
\begin{equation}\label{eq10}
A(z_1,\cdots, z_r) = a_\emptyset + \sum_{n=1}^\infty \sum_{i_1,\cdots,i_n =1}^r a_{i_1,\cdots,i_n}z_{i_1}\cdots z_{i_n},  
\end{equation}
where $a_\emptyset, a_{i_1,\cdots,i_n} \in \comp$ and $z_1,\cdots,z_r$ are free generators. Let us denote by $\cseries$ the set of such formal power series. Then we can define an associative product $\bullet$ as $F \bullet G:=S^{-1}((SF)\circ (SG))$, where 
$SF(z_1,\cdots,z_r):=(z_1F(z_1,\cdots,z_r), \cdots, z_rF(z_1,\cdots, z_r))$. 
This operation was essentially defined in \cite{H-S2}. For a given $A \in \cseries$ of the form (\ref{eq10}), we can 
associate $\widetilde{A} \in \text{\normalfont Mul}^r \llbracket \mathbb{C} \rrbracket$ by defining $\widetilde{A}_{i_1,\cdots,i_n}(b_1,\cdots,b_n):=a_{i_1,\cdots,i_n} b_1\cdots b_n$ ($b_i \in \comp$). Then the operation $\smiley$ coincides with $\bullet$. 
\item The operation $\star$ appears when we take a derivative of $F \circ H(t)$ regarding $\frac{d}{dt}|_0$ under the condition $H(0)=0$. 
\end{enumerate}
\begin{exa}
\begin{enumerate}[\rm(1)]
\item $(F \smiley G)_{i_1}(b_1) = F_{i_1}(G_\emptyset b_1)G_\emptyset + F_\emptyset G_{i_1}(b_1)$. 
\item $(F \smiley G)_{i_1,i_2}(b_1,b_2) = F_{i_1,i_2}(G_\emptyset b_1, G_\emptyset b_2)G_\emptyset + F_{i_1}(G_\emptyset b_1) G_{i_2}(b_2) +F_{i_2}(G_{i_1}(b_1)b_2)G_\emptyset + F_\emptyset G_{i_1,i_2} (b_1,b_2)$. 
\item $(F \smiley G)_{i_1,i_2,i_3}(b_1,b_2,b_3) = F_{i_1,i_2,i_3}(G_\emptyset b_1, G_\emptyset b_2, G_\emptyset b_3)G_\emptyset + F_{i_1,i_2}(G_\emptyset b_1, G_\emptyset b_2) G_{i_3}(b_3)$\\ $\,\,\,\,\,\,\,\,\,\,\,\,\,\,\,\,\,\,\,\,\,\,\,\,\,\,\,\,\,\,\,\,\,\,\,\,\,\,\,\,\,\,\,\,\,\,\,\,\,\,\,\,\,\,\,\,\,\,\,\,\,\,\,\,\,\,\,\,\,+F_{i_1, i_3}( G_\emptyset b_1, G_{i_2}(b_2)b_3)G_\emptyset +F_{i_2, i_3}(G_{i_1}(b_1)b_2, G_\emptyset b_3)G_\emptyset$ \\$\,\,\,\,\,\,\,\,\,\,\,\,\,\,\,\,\,\,\,\,\,\,\,\,\,\,\,\,\,\,\,\,\,\,\,\,\,\,\,\,\,\,\,\,\,\,\,\,\,\,\,\,\,\,\,\,\,\,\,\,\,\,\,\,\,\,\,\,\, +F_{i_1}(G_\emptyset b_1)G_{i_2,i_3}(b_2,b_3)  + F_{i_2}(G_{i_1}(b_1)b_2)G_{i_3}(b_3)$\\ $\,\,\,\,\,\,\,\,\,\,\,\,\,\,\,\,\,\,\,\,\,\,\,\,\,\,\,\,\,\,\,\,\,\,\,\,\,\,\,\,\,\,\,\,\,\,\,\,\,\,\,\,\,\,\,\,\,\,\,\,\,\,\,\,\,\,\,\,\,+ F_{i_3}(G_{i_1,i_2}(b_1,b_2)b_3)G_\emptyset + F_\emptyset G_{i_1,i_2, i_3} (b_1,b_2, b_3)$. 
\end{enumerate}

In general, $(F \smiley G)_{i_1,\cdots,i_n}(b_1,\cdots, b_n)$ can be written as the sum of $2^n$ terms. 
\end{exa}

\end{rem}
\begin{prop}
\begin{enumerate}[\rm(1)]
\item The composition $\smiley$ is associative: $(F\smiley G) \smiley H = F\smiley (G\smiley H)$ for any $F,G,H \in \mulb$. 
\item $\I:= (\I_{i_1,\cdots,i_n})$, defined by $\I_\emptyset = 1$ and $\I_{i_1,\cdots,i_n} =0$ for $n \geq 1$, is the identity for the operation $\smiley$: 
$\I \smiley F = F \smiley \I = F$ for any $F \in \mulb$.  
\item $(F+G) \smiley H = F\smiley H + G \smiley H$  for any $F,G,H \in \mulb$. 
\end{enumerate}
\end{prop}
\begin{proof}
(1) We fix a tuple $(i_1, \cdots, i_n)$ and we are going to prove the following:
\[
((F\smiley G) \smiley H)_{i_1,\cdots,i_n}(b_1,\cdots,b_{n})=(F\smiley (G\smiley H))_{i_1,\cdots,i_n}(b_1,\cdots,b_{n}). 
\]
For each sequence of words $\{F,G,H\}$ of length $n$, one can associate an $n$-linear functional as follows. Given a sequence $(A_1, \cdots, A_n)$, where each $A_j$ is equal to $F, G$ or $H$,  we gather the indices $j$ such that $A_j=F$ and denote them by $j_1 < j_2 < \cdots <j_p$.  
Let $J_s$ be the interpolation blocks ($s=1,\cdots, p+1$) of $J:=\{j_1, \cdots,j_p\}$ for $\underline{n}$. If $J$ is the empty set, then we understand that $p=0$ and $J_1=\underline{n}$. 
For each $s$, let us take all the indices $k \in J_s$ such that $A_k=G$ and denote them by $k_1^{(s)}< \cdots< k_{q(s)}^{(s)}$.  Let $K_m^{(s)}$ ($m=1,\cdots,q(s)+1$) be the interpolation blocks of $K^{(s)}:=\{k_1^{(s)}, \cdots, k_{q(s)}^{(s)}\}$ for the linearly ordered set $J_s$. 
Now we define 
\[
I_s(b_{J_s}) := 
\begin{cases}
\begin{split}
&G_{i(K^{(s)})}\left(H_{i(K_1^{(s)})}(b_{K_1^{(s)}})b_{k_1^{(s)}},  H_{i(K_2^{(s)})}(b_{K_2^{(s)}})b_{k_2^{(s)}}, \cdots,\right. \\
&~~~~~~~~~~~\left.H_{i(K_{q(s)}^{(s)})}(b_{K_{q(s)}^{(s)}})b_{k_{q(s)}^{(s)}}\right)H_{i(K_{q(s)+1}^{(s)})}(b_{K_{q(s)+1}^{(s)}}),  
\end{split}& \text{~if~}K^{(s)} \neq \emptyset, \\
G_\emptyset H_{i(K_{q(s)+1}^{(s)})}(b_{K_{q(s)+1}^{(s)}}),  & \text{~if~}K^{(s)} = \emptyset. 
\end{cases}
\]
We understand $H_{i(K_m^{(s)})}(b_{K_m^{(s)}}) = H_\emptyset \in \mathfrak{B}$ if $K_m^{(s)} = \emptyset$. 
Using $I_s$, we define an $n$-linear functional
$$
L_{(A_1,\cdots,A_n)}(b_1,\cdots,b_n):= 
\begin{cases}
F_{i_1,\cdots,i_n}(I_1(b_{J_1})b_{j_1}, \cdots, I_{p}(b_{J_{p}})b_{j_p}) I_{p+1}(b_{J_{p+1}}), &\text{~if~} J \neq \emptyset, \\
 F_\emptyset I_{p+1}(b_{J_{p+1}}),  &\text{~if~} J = \emptyset. 
 \end{cases}
$$
Examples of $L_{(A_1,\cdots,A_n)}$ can be found in the tables below.

\begin{table}[h]
\begin{minipage}{0.5\hsize}
\begin{center}
\begin{tabular}{|ll|} \hline
$(A_1)$ &  $L_{(A_1)}(b_1)$ \\ \hline
$(F)$ & $F_{i_1}(G_\emptyset H_\emptyset b_1)G_\emptyset H_\emptyset$  \\
$(G)$ & $F_{\emptyset}G_{i_1}(H_\emptyset b_1) H_\emptyset$ \\ 
$(H)$ &$F_\emptyset G_\emptyset H_{i_1}(b_1)$ \\ \hline
\end{tabular}
\caption{The elements of $\mathcal{L}_1$.}

\end{center}
  \end{minipage}
\begin{minipage}{0.5\hsize}
\begin{center}
\begin{tabular}{|ll|} \hline
$(A_1,A_2)$ &  $L_{(A_1,A_2)}(b_1,b_2)$ \\ \hline
$(F, F)$ & $F_{i_1,i_2}(G_\emptyset H_\emptyset b_1,G_\emptyset H_\emptyset b_2)G_\emptyset H_\emptyset$  \\
$(F, G)$ & $F_{i_1}(G_\emptyset H_\emptyset b_1)G_{i_2}( H_\emptyset b_2) H_\emptyset$ \\ 
$(F,H)$ &$F_{i_1}(G_\emptyset H_\emptyset b_1)G_\emptyset H_{i_2}(b_2)$ \\
$(G,F)$ & $F_{i_2}(G_{i_1}(H_\emptyset b_1)  H_\emptyset b_2) G_\emptyset  H_\emptyset$   \\
$(G,G)$ & $F_\emptyset G_{i_1,i_2}(H_\emptyset b_1, H_\emptyset b_2)H_\emptyset$   \\ 
$(G,H)$ &$F_\emptyset G_{i_1}(H_\emptyset b_1) H_{i_2}(b_2)$ \\
$(H,F)$ & $F_{i_2}(G_\emptyset H_{i_1}(b_1)b_2)G_\emptyset H_\emptyset$   \\
$(H,G)$ & $F_\emptyset G_{i_2}(H_{i_1}(b_1) b_2)H_\emptyset$ \\ 
$(H, H)$ & $F_\emptyset G_\emptyset H_{i_1,i_2}(b_1,b_2)$  \\ \hline
\end{tabular}
\caption{The elements of $\mathcal{L}_2$.}
\end{center}
\end{minipage}

\begin{center}
\begin{tabular}{|ll|} \hline
Word sequence & Multilinear functional \\ \hline
$(F, F, G)$ & $F_{i_1,i_2}(G_\emptyset H_\emptyset b_1,G_\emptyset H_\emptyset b_2)G_{i_3}(H_\emptyset b_3)H_\emptyset $  \\
$(F, G, F)$ & $F_{i_1,i_3}(G_\emptyset H_\emptyset b_1, G_{i_2}(H_\emptyset b_2)H_\emptyset b_3)G_\emptyset H_\emptyset $ \\ 
$(F,H,G, F)$ & $F_{i_1, i_4}(G_\emptyset H_\emptyset b_1, G_{i_3}(H_{i_2}(b_2)b_3)H_\emptyset b_4)G_\emptyset H_\emptyset $   \\
$(F,G, F, H)$ & $F_{i_1,i_3}(G_\emptyset H_\emptyset b_1, G_{i_2}(H_\emptyset b_2)H_\emptyset b_3)G_\emptyset H_{i_4}(b_4)$ \\
$(H,G, H,G)$ & $F_\emptyset G_{i_2, i_4}(H_{i_1}(b_1)b_2, H_{i_3}(b_3)b_4)H_\emptyset $   \\
$(H,G, F, G)$ & $F_{i_3}(G_{i_2}(H_{i_1}(b_1)b_2)H_\emptyset b_3)G_{i_4}(H_\emptyset b_4)H_\emptyset $ \\ \hline
\end{tabular}
\caption{Selected elements of $\mathcal{L}_n$ for $n =3,4$.}
\end{center}
\begin{center}
\begin{tabular}{|ll|}\hline
Word sequence & Multilinear functional \\ \hline
$(F,H,G, F,F)$ & $F_{i_1, i_4, i_5}(G_\emptyset H_\emptyset b_1, G_{i_3}(H_{i_2}(b_2)b_3)H_\emptyset b_4, G_\emptyset H_\emptyset b_5)G_\emptyset H_\emptyset $   \\
$(G,G, F, H,F,F)$ & $F_{i_3,i_5,i_6}(G_{i_1,i_2}(H_\emptyset b_1, H_\emptyset b_2)H_\emptyset b_3, G_\emptyset H_{i_4}(b_4)b_5, G_\emptyset H_\emptyset b_6)G_\emptyset H_\emptyset $ \\
$(G,H,H,F,G,F,H,H)$ & $F_{i_4,i_6}\left(G_{i_1}(H_\emptyset b_1)H_{i_2,i_3}(b_2,b_3)b_4,G_{i_5}(H_\emptyset b_5)H_\emptyset b_6\right)G_\emptyset H_{i_7,i_8}(b_7,b_8)$ \\ \hline
\end{tabular}
\caption{Selected elements of $\mathcal{L}_n$ for $n \geq 5$.}
\end{center}
\end{table}
Let us denote by $\mathcal{L}_n$ the set $\{L_{(A_1,\cdots,A_n)} \mid A_i \in \{F,G,H\}, 1\leq  i \leq n\}$, regarding the multilinear 
functionals $F_{i_1, \cdots,i_k}, G_{i_1,\cdots,i_k}, H_{i_1,\cdots,i_k}$ as indeterminates. Then the map $(A_1, \cdots, A_n) \mapsto L_{(A_1,\cdots,A_n)}$ is one-to-one to a set of indeterminates, and hence the cardinality of $\mathcal{L}_n$ is $3^n$. 
The multilinear functional  $((F\smiley G) \smiley H)_{i_1,\cdots,i_n}(b_1,\cdots,b_{n})$  can be expanded as the sum of distinct elements of $\mathcal{L}_n$ by definition, and moreover, the number of those elements is $\sum _i \binom{n-i}{i} 2^i=3^n$. 
So every element of $\mathcal{L}_n$ appears just once in $((F\smiley G) \smiley H)_{i_1,\cdots,i_n}(b_1,\cdots,b_{n})$, which means 
$((F\smiley G) \smiley H)_{i_1,\cdots,i_n}= \sum_{L \in \mathcal{L}_n}L$. 
A similar reasoning implies that $(F\smiley (G\smiley H))_{i_1,\cdots,i_n}= \sum_{L \in \mathcal{L}_n}L$. 
So we can conclude that $((F\smiley G) \smiley H)_{i_1,\cdots,i_n}=(F\smiley (G\smiley H))_{i_1,\cdots,i_n}$. 

Assertions (2) and (3) are not difficult. 
\end{proof}

\begin{defi}
\begin{enumerate}[\rm(1)]
\item We define a generating function of joint moments by 
\[
\mu^X:=(\mu^X _{i_1,\cdots,i_{n}})_{i_1,\cdots,i_{n} \in \underline{r}, n \geq 0} \in \mulb. 
\]
For $n =0$, $\mu^X_\emptyset$ is defined to be the unit  $1_\mathfrak{B}$. 
\item A generating function of cumulants $\kappa^X \in \mulb$ for $X=(X_1,\cdots,X_r)$ is defined by 
\[
\begin{cases}
\kappa^X_{i_1,\cdots,i_n} (b_1,\cdots,b_{n}) := K_n(b_1X_{i_1},\cdots,b_nX_{i_n}),& n \geq 1,\\ 
\kappa^X_\emptyset := 0. &
\end{cases}
\]
\end{enumerate}
\end{defi}

The following theorem extends formulae of \cite{Popa} and \cite{H-S2} to the sum of independent random vectors. 
\begin{thm}\label{thm11} (Extended Muraki's formula) 
Let $X=(X_1,\cdots, X_{r})$ and $Y=(Y_1,\cdots,Y_{r})$ be random vectors which are monotone independent over $\mathfrak{B}$. Then 
\[
\mu^{X+Y} = \mu^X \smiley \mu^Y. 
\]
\end{thm}
\begin{proof}
For simplicity, we assume that $(i_1,\cdots,i_n)=(1,\cdots,n)$ and prove that 
$\mu^{X+Y}_{1,\cdots,n} = (\mu^X \smiley \mu^Y)_{1,\cdots,n}$. The left hand side is  
\[
\begin{split}
\mu^{X+Y}_{1,\cdots,n}(b_1,\cdots,b_n)
&= \varphi(b_1(X_1+Y_1)b_2\cdots b_n(X_n + Y_n)) \\
&= \sum_{Z_i\in\{X_i,Y_i\},1\leq i \leq n}\varphi(b_1Z_1b_2\cdots b_nZ_n).  
\end{split}
\]
For each summand, let us denote by $V$ the positions where $X_i$ are taken, that is, $V=\{i:Z_i=X_i \}$, and write $V=\{v_1, \cdots, v_p\}$, $v_1 <\cdots <v_p$.  By using the interpolation blocks of $V$ for $\underline{n}$, the summand can be written as 
\begin{equation}\label{eq13}
\mu^X_{V}(\mu^Y_{V_1}(b_{V_1})b_{v_1},\cdots,\mu^Y_{V_p}(b_{V_p})b_{v_p})\mu^Y_{V_{p+1}}(b_{V_{p+1}}).
\end{equation}  
Therefore, $\mu^{X+Y}_{1,\cdots,n}(b_1,\cdots,b_n)$ is equal to the summation of (\ref{eq13}) over $V \subset \underline{n}$. By definition, this is equal to $(\mu^X \smiley \mu^Y)_{1,\cdots,n}$. 
\end{proof}

Using Lemma \ref{lem32}, one can define $\mu^{t.X}$ by extending $N$ of $\mu^{N.X}$ to $t \in \real$. Then we obtain differential equations.  
\begin{cor}
Let $X=(X_1,\cdots,X_r)$ be a random vector.  Then
\[
\frac{d}{dt}\mu^{t.X} = \kappa^X\smiley \mu^{t.X} = \mu^{t.X} \star \kappa^X.\]
\end{cor}
\begin{proof}
The equality $\mu^{(M+N).X} = \mu^{N.X} \smiley \mu^{M.X}$ follows from Theorem \ref{thm11}, with $X$ replaced by $X^{(1)}+\cdots+X^{(N)}$ and $Y$ by $X^{(N+1)}+\cdots+X^{(N+M)}$. Lemma \ref{lem32} implies that $\mu^{N.X}_{i_1,\cdots,i_n}$ is a polynomial of $N$, so that this is an identity as polynomials with respect to $N$ and $M$. Therefore, $N$ and $M$ can be replaced by real numbers $t$ and $s$, respectively. The derivatives $\frac{d}{dt}|_0$ and $\frac{d}{ds}|_0$ then yield the first and second identities, respectively.  
\end{proof}
The above two differential equations can be used to calculate monotone cumulants from moments. If $\mathfrak{B}=\comp$ and $r=1$, the above relations just coincide with 
\[
\frac{d}{dt}F^X_t(z) = A^X(F_t^X(z)) = A^X(z)\frac{\partial F_t^X}{\partial z}(z), 
\]  
where $F^X_t(z)$ is the reciprocal Cauchy transform of a `formal convolution semigroup' associated to $X$, and $A^X(z)=-\sum_{n=1}^\infty\frac{K_n(X)}{z^{n-1}}$ is a generating function of monotone cumulants. The reader is referred to the last remark of \cite{H-S2} for details.

\section*{Acknowledgements} TH thanks Mihai Popa for many discussions which improved his understanding of operator-valued independence. TH is supported by Global COE Program at Kyoto University.

\end{document}